\documentclass[12pt]{article}

\usepackage{tikz}

\usepackage{graphicx}
\usepackage{comment}
\usepackage{mathrsfs}
\usepackage{kpfonts}

\usepackage{physics}
\usepackage{nameref}
\usepackage[shortlabels,inline]{enumitem}
\usepackage{amsmath}
\usepackage{amssymb}
\usepackage{empheq}
\usepackage[shortlabels]{enumitem}
\setlist[enumerate]{nosep}
\usepackage{abstract}
\usepackage[doc]{optional}
\usepackage{xcolor}
\definecolor{lightgray}{gray}{0.9}
\usepackage[colorlinks=true,
            linkcolor=refkey,
            urlcolor=lblue,
            citecolor=red]{hyperref}
\usepackage{float}
\usepackage{soul}
\usepackage{graphicx}
\definecolor{labelkey}{rgb}{0,0.08,0.45}
\definecolor{refkey}{rgb}{0,0.6,0.0}
\definecolor{Brown}{rgb}{0.45,0.0,0.05}
\definecolor{lime}{rgb}{0.00,0.8,0.0}
\definecolor{lblue}{rgb}{0.5,0.5,0.99}

 \usepackage{mathpazo}

\usepackage{subcaption}
\colorlet{hlcyan}{cyan!30}

\usepackage{stmaryrd}
\usepackage{amssymb}

\usepackage{booktabs}
\usepackage[math]{cellspace}
\makeatletter
\def\namedlabel#1#2{\begingroup
   \def\@currentlabel{#2}%
   \label{#1}\endgroup
}
\makeatother

\newcommand{\sepp}{\setlength{\itemsep}{-2pt}}

\newcommand{\tu}{\ensuremath{\tilde{u}}}
\newcommand{\tv}{\ensuremath{\tilde{v}}}

\newcommand{\nnn}{\ensuremath{{n\in{\mathbb N}}}}
\newcommand{\thalb}{\ensuremath{\tfrac{1}{2}}}
\newcommand{\menge}[2]{\big\{{#1}~\big |~{#2}\big\}}

\newcommand{\fenv}[1]%
{\ensuremath{\,\overrightarrow{\operatorname{env}}_{#1}}}
\newcommand{\benv}[1]%
{\ensuremath{\,\overleftarrow{\operatorname{env}}_{#1}}}

\newcommand{\scal}[2]{\left\langle{#1},{#2}  \right\rangle}

\newcommand{\RR}{\ensuremath{\mathbb R}}
\newcommand{\Xt}{\ensuremath{{X}}}
\newcommand{\Ct}{\ensuremath{\tilde{C}}}

\newcommand{\Id}{\ensuremath{\operatorname{Id}}}

\newcommand{\proj}[1]{{\thinspace P\thinspace}_%
{\negthinspace\negthinspace #1}}

\newcommand{\minimize}[2]{\ensuremath{\underset{\substack{{#1}}}{\mathrm{minimize}}\;\;#2 }}

\DeclarePairedDelimiterX\set[2]{ \{ }{ \}_{#2} }{#1}

\DeclarePairedDelimiterX\rb[1]{ ( }{ ) }{#1}

{\begin{list}{}{%
\settowidth{\labelwidth}{\textrm{#1~}}%
\setlength{\leftmargin}{\labelwidth+\labelsep}}}
{\end{list}}
\usepackage{amsthm}
\usepackage[capitalize,nameinlink]{cleveref}
\crefname{equation}{}{equations}
\crefname{chapter}{Appendix}{chapters}
\crefname{item}{}{items}
\crefname{enumi}{}{}

\theoremstyle{definition}
\newtheorem{theorem}{Theorem}[section]
\newtheorem{lemma}[theorem]{Lemma}

\newtheorem{proposition}[theorem]{Proposition}

\newtheorem{definition}[theorem]{Definition}

\newtheorem{ex}[theorem]{Example}
\newtheorem{fact}[theorem]{Fact}
\newtheorem{remark}[theorem]{Remark}

\usepackage{multirow}

\providecommand{\norm}[1]{\lVert#1\rVert}

\providecommand{\lam}{\lambda}
\providecommand{\RR}{\mathbb{R}}
\providecommand{\proj}{\operatorname{Proj}}

\providecommand{\Id}{\operatorname{{ Id}}}

\providecommand{\Id}{\operatorname{Id}}

\providecommand{\RR}{\mathbb{R}}

\definecolor{myblue}{rgb}{.8, .8, 1}
  \newcommand*\mybluebox[1]{%
    \colorbox{myblue}{\hspace{1em}#1\hspace{1em}}}

\allowdisplaybreaks 

\usepackage{blkarray}
\usepackage{multirow}

\usepackage{amsmath}

\usetikzlibrary{calc,intersections}
\usepackage{pgfplots}
\usepackage[customcolors]{hf-tikz}

\tikzset{style green/.style={
    set fill color=green!50!lime!60,
    set border color=white,
  },
  style cyan/.style={
    set fill color=cyan!90!blue!60,
    set border color=white,
  },
  style gray/.style={
    set fill
    color=black!10!,
    set border color=white,
  },
  style orange/.style={
    set fill color=orange!80!red!60,
    set border color=white,
  },
  hor/.style={
    above left offset={-0.15,0.31},
    below right offset={0.15,-0.125},
    #1
  },
  ver/.style={
    above left offset={-0.1,0.3},
    below right offset={0.15,-0.15},
    #1
  }
}

\begin{document}

\title{\textsc{
Projections onto hyperbolas or bilinear constraint sets in Hilbert spaces
}}

\author{
Heinz H.\ Bauschke\thanks{
Mathematics, University
of British Columbia,
Kelowna, B.C. V1V~1V7, Canada. E-mail:
\texttt{heinz.bauschke@ubc.ca}.},~
Manish Krishan Lal\thanks{
Mathematics, University
of British Columbia,
Kelowna, B.C. V1V~1V7, Canada. E-mail:
\texttt{manish.krishanlal@ubc.ca}.},~
and
Xianfu Wang\thanks{
Mathematics, University
of British Columbia,
Kelowna, B.C. V1V~1V7, Canada. E-mail:
\texttt{shawn.wang@ubc.ca}.}
}
\date{December 1, 2021} 
\maketitle

\vskip 8mm

\begin{abstract} \noindent
Sets of bilinear constraints are important in various machine learning models. Mathematically, they are
hyperbolas in a product space.
In this paper, we give a complete formula for projections onto sets of bilinear constraints or
hyperbolas in a general Hilbert space.
\end{abstract}

{

}

\section{Introduction}
Throughout this paper, we assume that
\begin{empheq}[box=\mybluebox]{equation}
    \text{$X$ is a Hilbert space
    with inner product
    $\scal{\cdot}{\cdot}\colon X\times X\to\RR$, }
\end{empheq}
and induced norm $\|\cdot\|$. In various learning models \cite{elser2019learning,elser2017products}, Elser utilizes
projections onto the \emph{bilinear constraint} set:
\begin{empheq}[box=\mybluebox]{equation}
\label{e:defC}
C_\gamma := \menge{(x,y)\in X\times
X}{\scal{x}{y}=\gamma}
\end{empheq}
where $\gamma\in\RR$ is a fixed constant. Mathematically speaking, up to a rotation, this is just a hyperbola or
quadratic surface in $X\times X$. Finding projections onto quadratic curves or surfaces algorithmically
have many practical applications;
see, e.g., \cite{quadrics, conics, elser2019learning, elser2017products}.
If $\gamma=0$ in $C_{\gamma}$, then $C_0$ becomes the set `cross' in $X\times X$.
The projection formula for the `cross' has been thoroughly investigated \cite{cross}.
In \cite{elser2019learning,elser2017products}, although Elser has provided some
results on projections
onto $C_{\gamma}$ only when $X=\RR^n$, complete mathematical details are not presented.

\emph{The goal of this paper is to give a complete analysis of the projection onto the set $C_{\gamma}$ with $\gamma\neq 0$,
and extend the results to a general Hilbert space.}


\noindent The remainder of the paper is organized as follows. \cref{sec:aux} gives
some general properties of bilinear constraint sets. In \cref{section:hyperbola} we
focus on full mathematical details
for the existence and explicit formula of projections onto hyperbolas. Using results
from \cref{section:hyperbola}, in
Sections~\ref{sec:positive:gam}--\ref{sec:negative:gam} we provide explicit formulas for
projections onto sets of bilinear constraints.
Our notation is standard and follows largely \cite{BC2017,Rock98}.
\section{General properties of $C_{\gamma}$}
\label{sec:aux}
In an infinite-dimensional space, the existence of projection onto a set
often requires the weak closedness of the set.
Our first result says that
although the set $C_{\gamma}$ is norm closed it is not weakly closed in $X\times X$.
\begin{proposition}\label{p:weak}
The set $C_{\gamma}$ is closed in the
norm
topology but
not closed in the weak topology in $X\times X$.
In fact, $\overline{C_{\gamma}}^{\rm weak}=X\times X$.
\end{proposition}

\begin{proof} Evidently, $C_{\gamma}$ is norm closed. We show that
$\overline{C_{\gamma}}^{\rm weak}=X\times X$.
Let $(x,y)\in X\times X$. Consider $S=\mbox{span}\{x,y\}$.
The orthogonal decomposition theorem gives $X= S\oplus S^{\perp}$.
Because $X$ is infinite-dimensional and $S$ is at most two-dimensional, $S^{\perp}$ is infinite dimensional, so
any orthonormal base of $S^{\perp}$ must have a sequence $(e_{n})_{\nnn}$ which converges weakly to $0$, i.e.,
$e_{n}\rightharpoonup 0$; see, e.g., \cite{kreyszig}.
Also $e_{n}\perp x$ and
$e_{n}\perp y.$
Set $\xi =\gamma-\langle x, y\rangle$. Then
\begin{align}
\scal{x-\xi e_{n}}{y-e_{n}} &=
\scal{x}{y}-\scal{x}{e_{n}}-\xi\scal{e_{n}}{y}+\xi\scal{e_{n}}{e_{n}}\\
&=\scal{x}{y}+\xi  =\gamma,
\end{align}
so $(x-\xi e_{n}, y-e_{n})\in C_{\gamma}$.
Since $(x-\xi e_{n}, y-e_{n}) \rightharpoonup (x,y)$, we have
 $(x, y)\in \overline{C_{\gamma}}^{\rm weak}$.
Because $(x,y)\in X\times X$ was arbitrary, we conclude that
$X\times X\subseteq \overline{C}^{\rm weak}$.
\end{proof}

Proposition~\ref{p:weak} indicates that finding $P_{C_{\gamma}}$, i.e., projections onto $C_{\gamma}$,
might be complicated in a general Hilbert space.
This seemingly difficult issue can be completely avoided by utilizing the structures of
the optimization problem. See Subsection~\ref{s:existence} below.

Our next result says that when $\gamma\neq 0$, locally around
the set $C_{\gamma}$ the projection onto the set is always single-valued.
\begin{proposition}\label{p:regular}
 Let $\gamma\neq 0$, and $C_{\gamma}
= \menge{(x,y)\in X\times X}{\scal{x}{y}=\gamma}.$
For every $(x,y)\in C_{\gamma}$,
the the following hold:
\begin{enumerate}
\item\label{i:prox1}
$C_{\gamma}$ is prox-regular at $(x,y)$.
\item\label{i:prox2}
There exists a neighborhood of $(x,y)$
on which
the projection mapping onto $C_{\gamma}$ is monotone and Lipschitz continuous.
\end{enumerate}
\end{proposition}
\begin{proof}
\ref{i:prox1}: Write
$\iota_{C_{\gamma}}(x,y)=\iota_{\{\gamma\}}(h(x,y))$ where $h(x,y)=\scal{x}{y}$. Let
$(x,y)\in C_{\gamma}$. Then
$\nabla h(x,y)=(y,x) \neq (0,0)$ because of $\gamma\neq 0$.
 Being a composition of a convex function $\iota_{\{\gamma\}}$
and a twice
differentiable function $h$ that is qualified at $(x,y)$, \cite[Proposition 2.4]{thibault} shows that
$\iota_{C_{\gamma}}$ is prox-regular at $(x,y)$, so is $C_{\gamma}$.

\ref{i:prox2}: Apply \ref{i:prox1} and \cite[Proposition 4.4]{thibault}.
\end{proof}

Observe that  Proposition~\ref{p:regular}
also follow from \cite[Proposition 13.32]{Rock98} and \cite[Exercise 13.38]{Rock98} when $X=\RR^n$; and that
when $\gamma=0$, $C_{0}$ is not prox-regular at $(0,0)$. Although the projection
exists locally around $C_{\gamma}$, it is still not clear for the global existence.

\section{Projections onto hyperbolas}
\label{section:hyperbola}
For ease of analysis, we start with
\begin{empheq}{equation*}
C_{\gamma} := \menge{(x,y)\in X\times X}{h(x,y):=\scal{x}{y}-\gamma=0}
\end{empheq}
where $\gamma>0$.
Our goal is to find the projection formula
$P_{C_{\gamma}}(x_0,y_{0})$ for every $(x_{0},y_{0})\in X\times X$. That is,
\begin{empheq}{equation*}
\text{minimize} \quad f(x,y):=\|x-x_{0}\|^2+\|y-y_{0}\|^2
\quad\text{subject to} \quad (x,y) \in C_{\gamma}.\tag{P}
\end{empheq}

\subsection{Auxiliary problems and existence of projections}
\label{s:existence}
\noindent To determine the projection operator $\proj{C_{\gamma}}$ of the set $C_{\gamma}$,
we shall introduce two
equivalently
reformulated problems.
First, for every $(u_{0},v_{0})\in X\times X$, we solve the problem
\begin{empheq}{equation*}
\text{minimize} \quad f_{1}(u,v):=\|u-u_{0}\|^2+\|v-v_{0}\|^2
\quad\text{subject to} \quad (u,v) \in \Ct_1, \tag{$\tilde{P}_{1}$}
\end{empheq}
where  $\Ct_1:= \menge{(u,v)\in \Xt \times \Xt}{h_{1}(u,v):=\norm{u}^2 -\norm{v}^2-2=0}.$
Next, for every $(\tu_{0},\tv_{0})\in X\times X$ we solve the problem
\begin{empheq}{equation*}
\text{minimize} \quad f_{\gamma}(\tilde{u},\tilde{v}):=\|\tilde{u}-\tilde{u}_{0}\|^2+\|\tilde{v}-
\tilde{v}_{0}\|^2
\quad\text{subject to} \quad (\tilde{u},\tilde{v}) \in \Ct_{\gamma}, \tag{$\tilde{P}_{\gamma}$}
\end{empheq}
where $\Ct_{\gamma}:= \menge{(\tilde{u},\tilde{v})\in \Xt \times \Xt}{h_{\gamma}(
\tilde{u},\tilde{v}):=\norm{\tilde{u}}^2 -\norm{\tilde{v}}^2-2\gamma=0}$.
Both $\Ct_{1}$ and $\Ct_{\gamma}$ are hyperbolas.
$\tilde{P}_{1}$ and $\tilde{P}_{\gamma}$ solve for projections $P_{\Ct_{1}}$ and $P_{\Ct_{\gamma}}$ respectively,
and their connections to
$P_{C_{\gamma}}$ are given by the \cref{p:reformulate} below. Recall
\begin{definition}[Rotation with an angle $\phi$]
\label{f:rotation}
A change of coordinates $(u,v) \in \Xt \times \Xt$, by rotation through an angle $\phi$, is defined by
$$
\begin{bmatrix}x\\y\end{bmatrix}=\begin{bmatrix}\cos\phi\Id & -\sin\phi\Id \\
\addlinespace
\sin\phi\Id & \cos\phi\Id \end{bmatrix}
\begin{bmatrix}u\\v\end{bmatrix}.$$
\end{definition}
Put
$$A_{\phi}:=\begin{bmatrix}\cos\phi\Id & -\sin\phi\Id \\
\addlinespace
\sin\phi\Id & \cos\phi\Id \end{bmatrix}.$$
Then $A_{\phi}A_{-\phi}=\Id=A_{-\phi}A_{\phi}$ and $A^{-1}_{\phi}=A_{-\phi}$.
The relationships among $P_{C_{\gamma}}, P_{\Ct_{\gamma}}$ and $P_{\Ct_{1}}$ are summarized below.
\begin{proposition}\label{p:reformulate}
The following hold:
\begin{enumerate}
\item\label{i:scaling} $P_{\Ct_{\gamma}}=\sqrt{\gamma}P_{\Ct_{1}}(\Id/\sqrt{\gamma}).$
\item\label{i:rotate} $P_{C_{\gamma}}=A_{\pi/4}P_{\Ct_{\gamma}}A_{-\pi/4}.$
\end{enumerate}
\end{proposition}
\begin{proof}
\ref{i:scaling}: $(\tilde{P}_{1})$ is equivalent to $(\tilde{P}_{\gamma})$
by a change of variables of scaling
$\Id/\sqrt{\gamma}$. Indeed, using
$$
\begin{bmatrix}
u\\
v
\end{bmatrix}
=
\frac{1}{\sqrt{\gamma}}
\begin{bmatrix}
\tilde{u}\\
\tilde{v}
\end{bmatrix}, \text{ and } \begin{bmatrix}
u_{0}\\
v_{0}
\end{bmatrix}
=
\frac{1}{\sqrt{\gamma}}
\begin{bmatrix}
\tu_{0}\\
\tv_{0}
\end{bmatrix}
$$
$h_{\gamma}(\tilde{u},\tilde{v})=0$ becomes $h_{1}(u,v)=0$, and
$f_{\gamma}(\tilde{u},\tilde{v})$ becomes
$\gamma f_{1}(u,v)=\gamma(\|u-u_{0}\|^2+\|v-v_0\|^2)$.

\ref{i:rotate}: $(P)$ is equivalent to $(\tilde{P}_{\gamma})$ by a change of variables of
rational $A_{\pi/4}$. Indeed, with
$$\begin{bmatrix}
x\\
y
\end{bmatrix}
  =A_{\pi/4}\begin{bmatrix}
  \tilde{u}\\
  \tilde{v}
  \end{bmatrix}, \text{ and }
  \begin{bmatrix}
x_{0}\\
y_{0}
\end{bmatrix}
  =A_{\pi/4}\begin{bmatrix}
  \tilde{u}_{0}\\
  \tilde{v}_{0}
  \end{bmatrix}
  $$
 the objective $f(x,y) = \|x-x_0\|^2 + \|y-y_0\|^2=\|(x,y)-(x_{0},y_0)\|^2$ can be rewritten as
 \begin{align}
   f(x,y) &=
\norm{A_{\frac{\pi}{4}}\begin{bmatrix}
           \tilde{u}-\tilde{u}_0 \\
          \tilde{v}-\tilde{v}_0 \end{bmatrix}
     }^2 \nonumber =
     \begin{bmatrix}
           \tilde{u}-\tilde{u}_{0} \\
          \tilde{v}-\tilde{v}_{0} \end{bmatrix}^{T}A_{\frac{\pi}{4}}^{T}A_{\frac{\pi}{4}}\begin{bmatrix}
           \tilde{u}-\tilde{u}_{0} \\
          \tilde{v}-\tilde{v}_{0} \end{bmatrix} \\
& 
=\norm{\tilde{u}-\tilde{u}_0}^2 + \norm{\tilde{v}-\tilde{v}_0}^2, \nonumber\label{new_obj}
  \end{align}
  and
  $h(x,y)=\scal{x}{y}-\gamma=0$ becomes $h_{\gamma}(\tilde{u},\tilde{v})=
  \|\tilde{u}\|^2-\|\tilde{v}\|^2-2\gamma=0$.
\end{proof}

In view of Proposition~\ref{p:reformulate}\ref{i:rotate}, to see
that $P_{C_{\gamma}}(x,y)\neq \varnothing$ for every $(x,y)\in X\times X$,
the following observation is crucial.

\begin{proposition}\label{p:two:dim}
For every $(\tu_0,\tv_0) \in \Xt\times\Xt$,
the minimization problem
\begin{eqnarray}
\text{minimize} & \quad f_{\gamma}(\tu,\tv) =\|\tilde{u}-\tilde{u}_{0}\|^2+\|\tilde{v}-
\tilde{v}_{0}\|^2 \label{e:simple1}\\
\text{subject to} & \quad  h_{\gamma}(\tu,\tv) =\norm{\tilde{u}}^2 -\norm{\tilde{v}}^2-2\gamma =0
\label{e:simple2}
\end{eqnarray}
always has a solution, i.e., $P_{\Ct_{\gamma}}(\tu_{0},\tv_{0})\neq\varnothing$.
\end{proposition}

\begin{proof} We shall illustrate only the case $\tu_{0}\neq 0, \tv_{0}\neq 0$, since
the other cases are similar. We claim that the optimization problem is essentially
$2$-dimensional. To this end, we expand
\begin{align}
f_{\gamma}(\tu,\tv) &
=\underbrace{\|\tu\|^2-2\scal{\tu}{\tu_{0}}+\|\tu_{0}\|^2}+\underbrace{\|\tv\|^2-2\scal{\tv}{\tv_{0}}+\|\tv_{0}\|^2}.
\end{align}
The constraint
$$h_{\gamma}(\tu,\tv):=\norm{\tilde{u}}^2 -\norm{\tilde{v}}^2-2\gamma =0$$
means that only the norms $\|\tu\|$ and $\|\tv\|$ matter.
With $\|\tu\|$ fixed, the Cauchy-Schwarz inequality in a Hilbert space, see, e.g., \cite{kreyszig},
shows that $\tu\mapsto\scal{\tu}{\tu_{0}}$
will be larger so that
the first underlined part in $f_{\gamma}$ will be smaller when
$\tu$ and $\tu_{0}$ are positively co-linear, i.e, $\tu=\alpha\tu_{0}$ for some
$\alpha\geq 0$. Similarly, for fixed $\|v\|$ the second underlined part in
$f_{\gamma}$ will be smaller when
$\tv=\beta\tv_{0}$ for some $\beta\geq 0$. It follows that the optimization
problem given by \eqref{e:simple1}-\eqref{e:simple2} is equivalent to
\begin{align}
\text{minimize} &\quad g(\alpha,\beta):=(1-\alpha)^2\|\tilde{u}_{0}\|^2+(1-\beta)^2\|
\tilde{v}_{0}\|^2\label{e:sunday1}\\
\text{subject to} & \quad g_{1}(\alpha,\beta):=\alpha^2\norm{\tilde{u}_{0}}^2 -
\beta^2\norm{\tilde{v}_{0}}^2-2\gamma =0, \alpha\geq 0, \beta\geq 0.\label{e:sunday2}
\end{align}
Because $g:\RR^2\rightarrow\RR$ is continuous and coercive, and $g_{1}:\RR^2\rightarrow\RR$ is
continuous, we conclude that the optimization problem given by
\eqref{e:sunday1}-\eqref{e:sunday2} has a solution.
\end{proof}

We are now ready for the investigation of projections onto $\Ct_{1}, \Ct_{\gamma}$ and $C_{\gamma}$.

\subsection{Finding the projection $P_{\Ct_{1}}$}
Note that
\begin{equation}
\label{e:ngradients}
\nabla f_{1}(u,v) = (2(u-u_0),2(v-v_0))
\quad\text{and}\quad
\nabla h_{1}(u,v) = (2u,-2v).
\end{equation}
By \cite[Proposition~4.1.1]{Bertsekas}, every solution of ($\tilde{P}_{1}$) satisfies
the necessarily conditions, i.e., the KKT system of ($\tilde{P}_{1}$), given by
\begin{align}
(1+\lambda) u = u_0 \label{eq:l1}\\
(1 - \lambda) v = v_0 \label{eq:l2}\\ \|u\|^2-\|v\|^2-2=0\label{eq:l3}.
\end{align}
\begin{lemma}\label{l:bounds}
Let $u_0\neq 0,v_0\neq 0$. If $(u, v, \lambda)$ verifies \eqref{eq:l1}--\eqref{eq:l3} and
$(u,v)$ is an optimal solution of problem ($\tilde{P}_{1}$), then
the Lagrange multiplier $\lambda$
 satisfies $|\lambda|<1$.
\end{lemma}
\begin{proof}
The constraint set
$$\|u\|^2-\|v\|^2=2,$$
has a special structure: replacing $u$ by $-u$, or $v$ by $-v$
the constraint is still verified.
Also consider
$$f(u,v)=\|u-u_{0}\|^2+\|v-v_{0}\|^2=\|u_{0}\|^2-2\scal{u_{0}}{u}+
\|u\|^2+\|v_0\|^2-2\scal{v_{0}}{v}+\|v\|^2,$$
Given $u_{0} \text{ and }
v_{0}$, for fixed $\|u\|$ and $\|v\|$, $f(u,v)$ becomes smaller if one choose
$\scal{u_{0}}{u}\geq 0$, and
$\scal{v_{0}}{v}\geq 0$. Indeed, one can do so by replacing $u$ by $-u$ or $v$ by $-v$ if needed.
Now by \eqref{eq:l1} and \eqref{eq:l2},
$$(1+\lambda)\scal{u_{0}}{u}=\|u_{0}\|^2, \text{ and }
(1-\lambda)\scal{v_{0}}{v}=\|v_{0}\|^2.$$
Because $u_{0}\neq 0, v_{0}\neq 0$, we have $\scal{u_{0}}{u}>0, \scal{v_{0}}{v}>0$ so that
$1+\lambda>0, 1-\lambda>0$. Hence
$|\lambda|<1.$
\end{proof}

\begin{proposition}
\label{p:2011d4.1}
Let $u_0\neq0, v_0\neq 0$. Define $p \coloneqq\norm{u_0}^2 - \norm{v_0}^2 $ and $q \coloneqq\norm{u_0}^2 + \norm{v_0}^2$.
Suppose that $u,v\in\Xt$ and $\lambda\in\RR$ verify \eqref{eq:l1}-\eqref{eq:l3} and $(u,v)$ is an optimal solution
to ($\tilde{P}_{1}$).
Then
the following hold:
\begin{enumerate}
\item
\label{p:2011411}
$u = \frac{u_0}{(1+\lambda)}$, $v = \frac{v_0}{(1-\lambda)}$, and
$(1-\lambda)^2\norm{u_0}^2 - (1+\lambda)^2\norm{v_0}^2
= 2(1-\lambda^2)^2$.

\item
\label{p:2011412}
The objective function has
$
f_{1}(u,v) = \lambda^2\bigg(\frac{\|u_0\|^2}{(1+\lambda)^2}+\frac{\|v_0\|^2}{(1-\lambda)^2} \bigg).
$
\item
\label{p:2011413}
$\lambda$ is the unique solution of
\begin{equation}
\label{eq:main}
H(\lambda) \coloneqq \frac{(\lambda^2 + 1 ) p - 2\lambda q}{2(1-\lambda^2)^2} -1  = 0
\end{equation}
in $\left]-1,1\right[$.
\end{enumerate}
\end{proposition}
\begin{proof}
\cref{p:2011411}: Because $u_0\neq0, v_0\neq0$, we obtain $|\lambda|<1$,
$u = \frac{u_0}{(1+\lambda)}$ and $v = \frac{v_0}{(1-\lambda)}$ by
Lemma~\ref{l:bounds} and \eqref{eq:l1}-\eqref{eq:l2}.
Then we have the following equivalences
\begin{align}
\norm{u}^2 - \norm{v}^2 =2
&\Leftrightarrow
\frac{\norm{u_0}^2}{(1+\lambda)^2} - \frac{\norm{v_0}^2}{(1-\lambda)^2} = 2 \\
&\Leftrightarrow
(1-\lambda)^2\norm{u_0}^2  - (1+\lambda)^2 \norm{v_0}^2 =2(1-\lambda)^2(1+\lambda)^2\\
&\Leftrightarrow
(1-\lambda)^2\norm{u_0}^2  - (1+\lambda)^2 \norm{v_0}^2 =2(1-\lambda^2)^2.
\end{align}
\cref{p:2011412}: Substitute $u = \frac{u_0}{(1+\lambda)}$ and $v = \frac{v_0}{(1-\lambda)}$ in $f_{1}$.\\
\cref{p:2011413}:
By \cref{p:2011411}, we have
\begin{equation}
\label{e:2011416}
(1-\lambda)^2\norm{u_0}^2 - (1+\lambda)^2 \norm{v_0}^2 =2(1-\lambda^2)^2.
\end{equation}
Since
\begin{align}
&\hspace{-1cm}
(1-\lambda)^2\norm{u_0}^2 - (1+\lambda)^2 \norm{v_0}^2\\
&=
\big(1+\lambda^2-2\lambda\big)\|u_0\|^2 - \big(1+\lambda^2+2\lambda\big)\|v_0\|^2\\
&=
\lambda^2\big(\|u_0\|^2-\|v_0\|^2\big)+\big(\|u_0\|^2-\|v_0\|^2\big)-2\lambda\big(\|u_0\|^2+\|v_0\|^2\big) \\
&=\lambda^2 p+p-2\lambda q, \label{e:num}
\end{align}
using \cref{e:num} on left side of \cref{e:2011416}, we obtain $\lambda^2 p+p-2\lambda q =2(1-\lambda^2)^2$,
a univariate quartic equation in $\lam$, equivalently,
\begin{align}
\label{eq:quartic1}
 H(\lambda) \coloneqq \frac{(\lambda^2 + 1 ) p - 2\lambda q}{2(1-\lambda^2)^2} -1  = 0.
\end{align}
We show that \cref{eq:quartic1} has a unique solution in $]-1,1[$.
Because $u_0\neq0, v_0\neq0$,
we know that $q >|p|$,
and using it along-with $|\lambda|<1$ from \cref{l:bounds},
we get $p\lambda \leq |p||\lambda|\leq q|\lambda|$, and
\begin{align}
H'(\lambda)
&= \frac{1}{(1-\lambda^2)^3}\big(-q(1+3\lambda^2) + p(\lambda^3+3\lambda)\big)\\
&\leq\frac{1}{(1-\lambda^2)^3}\big(-q(1+3\lambda^2) + |p||\lambda|(\lambda^2+3)\big)\\
&\leq\frac{1}{(1-\lambda^2)^3}\big(-q(1+3\lambda^2)+ q|\lambda|(\lambda^2+3)\big)\\
&= \frac{q}{(1-\lambda^2)^3}\big((|\lambda|-1)^3\big)
\\
&= \frac{-q}{((1-|\lambda|)(1+|\lambda|))^3}\big((1-|\lambda|)^3\big)
\\ &= \frac{-q}{(1+|\lambda|)^3}.
\end{align}
Since $q = \|u_0\|^2 + \|v_0\|^2 > 0$ and $(1+|\lambda|)^3 > 0$, we have $H'(\lambda)<0$ on $]-1,1[$, which implies  $H(\lambda)$
is strictly decreasing on $]-1,1[$. Notice that $\lambda = \pm 1$ are vertical asymptotes because
at $\lambda = \pm 1$ the numerator of $H(\lambda)$ is, $H(-1) = 2(p+q)=4\|u_0\|^2 > 0$ and $H(1) = 2(p - q) = - 4 \|v_0\|^2 < 0$, so
that $\lim_{\lambda \to 1^{-}} H(\lambda) = - \infty$ and $\lim_{\lambda \to -1^{+}} H(\lambda) = + \infty$. Since
$H$ is continuous, strictly decreasing, and its range is $]-\infty, \infty[$, we conclude
that $H(\lambda) = 0$ has a unique zero in $]-1,1[$.
\end{proof}

\begin{remark}
As indicated in \cite{elser2019learning},
an approximate solution to \eqref{eq:main} can be found by the Bisection method, Newton's method, or a combined
version. See also \cite{burden2011analisis, robust}.
\end{remark}

\begin{theorem}
\label{t:hyperbola1}
Let $u_0,v_0 \in \Xt$, $u_{0} \neq 0, v_{0} \neq 0$. Then $\proj{\Ct_1}(u_0,v_0)$ is a singleton, and
 $$\proj{\Ct_1}(u_0,v_0)
  =
 \left\{ \begin{bmatrix}
   \frac{u_0}{1+\lambda}\\
   \addlinespace
   \frac{v_0}{1-\lambda}
 \end{bmatrix}
 \right\}
  $$
 in which $\lambda$ is the unique root of $H(\lambda)=0$ in $]-1,1[$ where
  \begin{align*}
     H(\lambda) =
      \frac{(\lambda^2 + 1 ) p - 2\lambda q}{2(1-\lambda^2)^2} -1, \text{ and }
        \\
      p =\norm{u_0}^2 - \norm{v_0}^2, \quad q =\norm{u_0}^2 + \norm{v_0}^2.
      \end{align*}
\end{theorem}
\begin{proof}
Apply \cref{p:2011d4.1}.
\end{proof}
\begin{theorem}
\label{p:2011d4.5}
  Let $u_0,v_0 \in \Xt$ with either $u_0=0$ or $v_0=0$.
\begin{enumerate}
\item
\label{p:2011d4.5a} When $u_0 = 0,$
we have
    \begin{align}
    \label{f:1}
    \proj{\Ct_1}(0,v_0) = \left\{
    \begin{bmatrix}
    u\\
    \frac{v_0}{2}
    \end{bmatrix}
    \bigg|\ \|u\|^2 = 2+ \frac{\|v_0\|^2}{4}
    \right\}.\end{align}
    \item
    \label{p:2011d4.5b}
    a) When $ v_0 = 0$ and $\|u_0\|\geq 2\sqrt{2}$, we have
    \begin{align}
    \label{f:2a}
    \proj{\Ct_1}(u_0,0) = \left\{
    \begin{bmatrix}
    \frac{u_0}{2}\\
    v\end{bmatrix}
    \bigg| \ \|v\|^2 = \frac{\|u_0\|^2}{4}-2\right\}.
    \end{align}
    b) When $v_0 = 0$ and $ 0<\|u_0\|< 2\sqrt{2}$, we have
    \begin{align}
    \label{f:2b}
    \proj{\Ct_1}(u_0,0) = \left\{\begin{bmatrix}
    \sqrt{2}\frac{u_0}{\|u_0\|}\\
    0\end{bmatrix}
    \right\}.\end{align}
\end{enumerate}
\end{theorem}

\begin{proof} By \eqref{eq:l1}--\eqref{eq:l3}, $(u,v)\in \proj{\Ct_1}(u_{0},v_0)$ satisfies:
\begin{align}
        (1+\lambda) u = u_0 \label{0.1}\\
        (1 - \lambda) v = v_0 \label{0.2}\\
        \|u\|^2 - \|v\|^2=2    \label{0.3},
\end{align}
for some $\lambda\in\RR$.

\noindent \ref{p:2011d4.5a}: $u_0 = 0,v_0\neq0$.
\eqref{0.1} yields
 $(1+\lambda)u=0 \Rightarrow$ either $1+\lambda=0$ or $u=0$. We consider two cases.

\textbf{Case 1:}~$1+\lambda=0$, i.e., $\lambda =-1$.
By (\ref{0.2}),
$2v =v_0 \Rightarrow v = \frac{v_0}{2}$ and then
by (\ref{0.3}), $\|u\|^2-\|\frac{v_0}{2}\|^2=2 \Rightarrow \|u\|^2=2+\frac{\|v_0\|^2}{4}$. So, $(u,v)=\left(u,\frac{v_0}{2}\right)$ and
$\|u\|^2=2+\frac{\|v_0\|^2}{4}$. The objective function is
\begin{align*}
    f(u,v) =\|u-u_0\|^2+\|v-v_0\|^2
= \|u\|^2+\norm{\frac{v_0}{2}}^2
    = 2+\frac{\|v_0\|^2}{4}+\frac{\|v_0\|^2}{4}
    = 2+\frac{\|v_0\|^2}{2}.
\end{align*}

\textbf{Case 2:}~$u=0$. By (\ref{0.3}), $(0-\|v\|^2)=2$, which is impossible.\\
Combining both cases we have the formula (\ref{f:1}).

\noindent \ref{p:2011d4.5b}: $u_0 \neq 0$, but $v_0 = 0$, which implies $(1-\lambda)v =0$ by (\ref{0.2}).
We consider two cases:

\textbf{Case 1:} $1-\lambda=0$, i.e., $ \lambda = 1$.
By (\ref{0.1}), $2u = u_0 \Rightarrow u = \frac{u_0}{2}$ and then by (\ref{0.3}), $\|v\|^2 = \frac{\|u_0\|^2}{4}-2 \Rightarrow \frac{\|u_0\|^2}{4}-2 \geq 0 \Rightarrow \|u_0\|^2 \geq 8$. Thus, $(u,v)=\left(\frac{u_0}{2},v\right)$ with
$\|v\|^2=\frac{\|u_0\|^2}{4}-2$, where the objective is
\begin{align*}
    f(u,v) = \|u-u_0\|^2+\|v-v_0\|^2
= \frac{\|u_0\|^2}{4}+\frac{\|u_0\|^2}{4} -2
    = \frac{\|u_0\|^2}{2} -2.
\end{align*}
Note that Case 1 needs $\|u_0\|\geq 2\sqrt{2}$.

\textbf{Case 2:} $v=0$. Then $\|u\|^2=2 \Rightarrow \|u\|=\sqrt{2}$
by \eqref{0.3}. By (\ref{0.1}), $|1+\lambda|\|u\|=\|u_0\| \Rightarrow |1+\lambda| = \frac{\|u_0\|}{\sqrt{2}} \Rightarrow 1+\lambda = \pm \frac{\|u_0\|}{\sqrt{2}}.$
\\ \textbf{Subcase 1:}
$1+\lambda = -\frac{\|u_0\|}{\sqrt{2}}$. By (\ref{0.1}), $u = \frac{u_0}{(-\frac{\|u_0\|}{\sqrt{2}})} = - \frac{\sqrt{2} u_0}{\|u_0\|}$. Then
$$f(u,0)=\|u-u_0\|^2+\|0-v_0\|^2= \norm{-\frac{\sqrt{2} u_0}{\|u_0\|}-u_0}^2
    = 
    (\|u_0\|+\sqrt{2})^2.$$
\\ \textbf{Subcase 2:}
$1+\lambda = \frac{\|u_0\|}{\sqrt{2}}.$
By \eqref{0.1},
$u = \frac{\sqrt{2} u_0}{\|u_0\|}$. Then
$$f(u,0)=\|u-u_0\|^2+\|0-v_0\|^2= \norm{u_0-\frac{\sqrt{2} u_0}{\|u_0\|}}^2
= (\|u_0\|-\sqrt{2})^2. $$
Comparing Subcase 1 and Subcase 2, we obtain $|\|u_0\|-\sqrt{2}|<\|u_0\|+\sqrt{2}$ because $\|u_0\|\neq 0$.
That is, Subcase 2 gives a smaller value at $(u,0)$ with $u=\sqrt{2}\frac{u_0}{\|u_0\|}$. We need to compare it
to Case 1 whenever it happens. Note that
$$\frac{\|u_0\|^2}{2}-2<(\|u_0\|-\sqrt{2})^2\quad\text{whenever}\quad\|u_0\|\neq 2\sqrt{2}.$$
Indeed, we have
$$
    \frac{\|u_0\|^2}{2}-2<\|u_0\|^2-2\sqrt{2}\|u_0\|+2 \Longleftrightarrow \frac{\|u_0\|^2}{2} -2\sqrt{2}\|u_0\|+4 > 0\\\Longleftrightarrow
    (\|u_0\|-2\sqrt{2})^2 >0,
$$
which always holds if $\|u_0\|\neq 2\sqrt{2}$. Hence, the nearest points are given by
\begin{align}
\label{smile1}
    \left(\frac{u_0}{2},v\right)\quad
    \text{with}\quad \|v\|^2 = \frac{\|u_0\|^2}{4}-2,\quad\text{ when}\quad \|u_0\|\neq 2\sqrt{2}
\end{align}
\begin{align}
\label{smile2}
\left(\frac{u_0}{2},0\right)
=\left(\sqrt{2}\frac{u_0}{\|u_0\|},0\right), \quad\text{ when}\quad \|u_0\|=2\sqrt{2}.
\end{align}
However, Case 1 occurs only when
$\|u_0\|\geq 2\sqrt{2}$. Hence, when
$\|u_0\|\geq 2\sqrt{2}$ the nearest points are given by
$$\left(\frac{u_0}{2},v\right)
\quad \text{with} \quad \|v\|^2 = \frac{\|u_0\|^2}{4}-2.$$
When $\|u_0\|<2\sqrt{2}$, Case 1 is impossible.
Then we only need to compare Subcase 1 and Subcase 2. Hence, the nearest point is
$$\left(\sqrt{2}\frac{u_0}{\|u_0\|},0\right)\quad\text{when}\quad\|u_0\|<2\sqrt{2}.$$
Finally, since $\|u_0\|=2\sqrt{2} \Rightarrow \|v\|^2=2-2=0$, (\ref{smile1}) gives $\left(\frac{u_0}{2},
0\right)$. Also $\frac{\sqrt{2}}{\|u_0\|}=\frac{1}{2}$, therefore $\sqrt{2}\frac{u_0}{\|u_0\|}=\thalb u_0$.
It follows that
$\left(\frac{u_0}{2},0\right)=\left(\sqrt{2}\frac{u_0}{\|u_0\|},0\right)$
when $\|u_0\|=2\sqrt{2}$, which implies that
(\ref{smile2}) can be obtained from (\ref{smile1}) when $\|u_0\|=2\sqrt{2}$.
Hence formulas (\ref{f:2a}) and (\ref{f:2b}) hold.
\end{proof}


\subsection{Finding the projection $P_{\Ct_{\gamma}}$}

\noindent $\proj{\Ct_{\gamma}}$ can be found via $\proj{\Ct_1}$.

\begin{theorem}
\label{t:201147}
Let $\gamma>0, \tilde{u}_0,\tilde{v}_0 \in \Xt$, and
$\Ct_{\gamma} =
      \menge{(\tilde{u},\tilde{v})\in \Xt \times \Xt}{\norm{\tilde{u}}^2 -\norm{\tilde{v}}^2 = 2\gamma}$. Then the following hold:
\begin{enumerate}
\item \label{t:201147.a} When $\tilde{u}_0 \neq 0$ and $ \tilde{v}_0 \neq 0$,
we have
$$\proj{\Ct_{\gamma}}(\tilde{u}_0,\tilde{v}_0)
=\left\{\begin{bmatrix}
   \frac{\tilde{u}_0}{1+\lambda}\\
   \addlinespace
   \frac{\tilde{v}_0}{1-\lambda}
  \end{bmatrix}
  \right\},
  $$
in which $\lambda$ is the unique root of $H(\lambda)=0$ in $\left]-1, 1\right[$, where
  \begin{align*}
     H(\lambda) =
      \frac{(\lambda^2 + 1 ) p - 2\lambda q}{2(1-\lambda^2)^2} -\gamma,~
      p =\norm{\tilde{u}_0}^2 - \norm{\tilde{v}_0}^2,~ q =\norm{\tilde{u}_0}^2 + \norm{\tilde{v}_0}^2.
      \end{align*}
\item \label{t:201147.b} When $\tilde{u}_0 = 0,$
we have
    \begin{align}
    \label{k:1}
    \proj{\Ct_{\gamma}}(0,\tilde{v}_0) = \left\{
    \begin{bmatrix}\tilde{u}\\
    \frac{\tilde{v}_0}{2}\end{bmatrix}
    \bigg|\ \|\tilde{u}\|^2 =
    2\gamma+ \frac{\|\tilde{v}_0\|^2}{4}\right\}.\end{align}
    \item
    \label{t:201147.c}
    a) When $ \tilde{v}_0 = 0$ and $ \|\tilde{u}_0\|\geq 2\sqrt{2\gamma}$, we have
    \begin{align}
    \label{k:2a}
    \proj{\Ct_{\gamma}}(\tilde{u}_0,0) = \left\{
    \begin{bmatrix}
    \frac{\tilde{u}_0}{2}\\
    \tilde{v}
    \end{bmatrix}
    \bigg| \ \|\tilde{v}
    \|^2 = \frac{\|\tilde{u}_0\|^2}{4}-2\gamma\right\}.
    \end{align}
    b) When $ \tilde{v}_0 = 0$ and
    $0<\|\tilde{u}_0\|< 2\sqrt{2\gamma}$, we have
    \begin{align}
    \label{k:2b}
    \proj{\Ct_{\gamma}}(\tilde{u}_0,0) = \left\{
    \begin{bmatrix}
    \sqrt{2\gamma}\frac{\tilde{u}_0}{\|\tilde{u}_0\|}\\
    0
    \end{bmatrix}
    \right\}.
    \end{align}
\end{enumerate}
\end{theorem}
\begin{proof}
Proposition~\ref{p:reformulate}\ref{i:scaling} states
$$\proj{\Ct_{\gamma}}(\tilde{u}_0,\tilde{v}_0)=\sqrt{\gamma}\proj{\Ct_1}\left(
\frac{\tilde{u}_{0}}{\sqrt{\gamma}},\frac{\tilde{v}_{0}}{\sqrt{\gamma}}\right).$$
Apply \cref{t:hyperbola1} and \cref{p:2011d4.5}.
\end{proof}

\section{Projections onto bilinear constraint set $C_{\gamma}$ when $\gamma>0$}
\label{sec:positive:gam}
$P_{C_{\gamma}}$ can be found via $P_{\Ct_{\gamma}}$, which is the main result of this section.
\begin{theorem}
\label{t:withgamma}
Let $\gamma>0,~x_0,y_0 \in X$, and $C_{\gamma} =
      \menge{(x,y)\in X \times X}{\scal{x}{y} = \gamma}$.  
Then the following hold:
\begin{enumerate}
    \item \label{i:xy:notbothe}
    When $x_0 \neq  \pm y_0$, the projection is a singleton:
    \begin{align*}
\proj{C_{\gamma}}(x_0,y_0)
          =\left\{\begin{bmatrix}
    \frac{x_0-\lam y_0}{1-\lam^2} \\
    \addlinespace
    \frac{y_0-\lam x_0}{1-\lam^2}
    \end{bmatrix}\right\},
    \end{align*}
    in which $\lambda$ is the unique solution of $H(\lambda)=0$ in $\left]-1,1\right[$, where
    \begin{align*}
      H(\lambda) =
      \frac{(\lambda^2 + 1 ) p - 2\lambda q}{2(1-\lambda^2)^2} -\gamma,~
      p =2\scal{x_0}{y_0},~\text{ and } q =\norm{x_0}^2 + \norm{y_0}^2.
  \end{align*}
   \item\label{i:u:zero}
    When $x_0 =-y_0$,  the projection is a set:
    $$\proj{C_{\gamma}}(x_0,-x_0) =\left\{
  \begin{bmatrix}
  \frac{x_{0}}{2}+\frac{\tu}{\sqrt{2}}\\
  \addlinespace
  -\frac{x_0}{2}+\frac{\tu}{\sqrt{2}}
   \end{bmatrix}
  \bigg|\ \|\tu\|^2=2\gamma+ \frac{\|x_0\|^2}{2},~~ \tu \in X \right\}.$$

    \item\label{i:v:zero1}
    a) When $x_0 =y_0$ and $\|x_0\|\geq 2\sqrt{\gamma}$,
    the projection is a set:
    $$\proj{C_{\gamma}}(x_0,x_0) =  \left\{
   \begin{bmatrix}
   \frac{x_0}{2}-\frac{\tv}{\sqrt{2}}\\
   \addlinespace
   \frac{x_0}{2}+\frac{\tv}{\sqrt{2}}
   \end{bmatrix}
   \bigg|\ \|\tv\|^2=\frac{\|x_0\|^2}{2}-2\gamma,~~ \tv \in X \right\}.$$
  b) When $x_0 =y_0$ and $0<\|x_0\|< 2\sqrt{\gamma}$, the projection
  is a singleton:
    $$\proj{C_{\gamma}}(x_0,x_0) =  \left\{\sqrt{\gamma}\begin{bmatrix}\frac{x_{0}}{\|x_{0}\|}\\
    \addlinespace
    \frac{x_{0}}{\|x_{0}\|}
    \end{bmatrix}\right\}.
    $$
   \end{enumerate}
\end{theorem}
\begin{proof} Proposition~\ref{p:reformulate}\ref{i:rotate} states
\begin{equation}\label{e:pformula}
\proj{C_{\gamma}}(x_0,y_0) = A_{\frac{\pi}{4}}
  \proj{\Ct_{\gamma}}A_{-\frac{\pi}{4}}(x_{0},y_{0}).
    \end{equation}
It suffices to apply \cref{t:201147}. Indeed, with
\begin{equation}\label{e:point}
\begin{bmatrix}
\tilde{u}_{0}\\
\tilde{v}_{0}
\end{bmatrix} =A_{-\pi/4}\begin{bmatrix}
x_{0}\\
y_{0}
\end{bmatrix}=\begin{bmatrix}
\frac{x_{0}+y_{0}}{\sqrt{2}}\\
\addlinespace
\frac{-x_{0}+y_{0}}{\sqrt{2}}
\end{bmatrix},
\end{equation}
using \eqref{e:pformula} and \eqref{e:point} Theorem~\ref{t:201147} gives:
\bigskip

\ref{i:xy:notbothe}:  When $x_{0}\neq \pm y_{0}$, we have $\tu_{0}\neq 0, \tv_{0}\neq 0$.
By \cref{t:201147}\cref{t:201147.a}, we get
$\proj{\Ct_{\gamma}}(\tu_0,\tv_0)= [\frac{\tu_0}{(1+\lam)},\frac{\tv_0}{(1-\lam)}]^T$, so
\begin{align*}
\proj{C_{\gamma}}(x_0,y_0) &= \left[ {\begin{array}{cc}
   \frac{1}{\sqrt{2}}\Id & -\frac{1}{\sqrt{2}}\Id\\
   \addlinespace
   \frac{1}{\sqrt{2}}\Id & \frac{1}{\sqrt{2}}\Id \\
  \end{array} } \right]\left[ {\begin{array}{cc}
   \frac{1}{1+\lambda}\Id & 0\\
   \addlinespace
   0 & \frac{1}{1-\lambda}\Id \\
  \end{array} } \right]\left[ {\begin{array}{cc}
   \frac{1}{\sqrt{2}}\Id & \frac{1}{\sqrt{2}}\Id\\
   \addlinespace
   -\frac{1}{\sqrt{2}}\Id & \frac{1}{\sqrt{2}}\Id \\
  \end{array} } \right]\begin{bmatrix}
           x_0 \\
          y_0 \end{bmatrix}\\&=  \left[ {\begin{array}{cc}
   \frac{1}{1-\lambda^2}\Id & \frac{-\lambda}{1-\lambda^2}\Id\\
   \addlinespace
   \frac{-\lambda}{1-\lambda^2}\Id & \frac{1}{1-\lambda^2}\Id \\
  \end{array} } \right]\left[ {\begin{array}{cc}
   x_0\\
   y_0
  \end{array} } \right].
 \end{align*}
 Also,
 $\|p\|=\|\tu_{0}\|^2-\|\tv_{0}\|^2=2\scal{x_{0}}{y_{0}},$ and
 $\|q\|=\|\tu_{0}\|^2+\|\tv_{0}\|^2=\|x_{0}\|^2+\|y_{0}\|^2$.

  \ref{i:u:zero}:
  For $x_0=-y_0$, we have $\tu_{0}=0, \tv_{0}=-\sqrt{2}x_{0}$, so
  \begin{align*}
      \proj{C_{\gamma}}(x_0,-x_0)&=A_{\frac{\pi}{4}}\proj{\Ct_{\gamma}
      }(0,-\sqrt{2}x_0)\\
  &=\left\{A_{\frac{\pi}{4}}\begin{bmatrix}
  \tu\\
  \addlinespace
  \frac{-x_0}{\sqrt{2}}
  \end{bmatrix}\bigg|\|\tu\|^2=2\gamma+\frac{\|-\sqrt{2}x_0\|^2}{4}=2\gamma+\frac{\|x_0\|^2}{2}, \tu\in X\right\}.
  \end{align*}

  \ref{i:v:zero1}:
  When $x_0\neq -y_0$, and $x_0=y_0$, we have $x_0\neq 0, (x_0,y_0)=(x_0,x_0)$, so that
  $\tv_{0}=0$ and
  $\tu_0 = \frac{x_0+y_0}{\sqrt{2}}=\sqrt{2}x_0.$  Then
  $\|\tu_0\|\geq 2\sqrt{2\gamma} \Leftrightarrow \|x_0\|\geq 2\sqrt{\gamma}.$

  a) When $\|x_{0}\|\geq 2\sqrt{\gamma}$, we have
  \begin{align*}
    \proj{C_{\gamma}}(x_0,x_0)&=A_{\frac{\pi}{4}}\proj{\Ct_{\gamma}}(\sqrt{2}x_0,0)\\
    &=A_{\frac{\pi}{4}}\left\{\begin{bmatrix}
  \frac{x_0}{\sqrt{2}}\\
  \addlinespace
  \tv
  \end{bmatrix}\bigg|\ \|\tv\|^2=\frac{\|\sqrt{2}x_0\|^2}{4}-2\gamma=\frac{\|x_0\|^2}{2}-2\gamma,
  \tv\in X \right\}.
  \end{align*}

  b) When $0<\|x_0\|<2\sqrt{\gamma}$, we have $\|\tu_0\|<2\sqrt{2\gamma}$, so
  \begin{align*}
      \proj{C_{\gamma}}(x_0,x_0)&=A_{\frac{\pi}{4}}\proj{\Ct_{\gamma}}
      (\sqrt{2}x_0,0)=A_{\frac{\pi}{4}}\begin{bmatrix}
  \sqrt{2\gamma} \frac{\sqrt{2} x_0}{\|\sqrt{2}x_0\|}\\
  \addlinespace
  0
  \end{bmatrix}
  =A_{\frac{\pi}{4}}
  \begin{bmatrix}
  \sqrt{2\gamma} \frac{x_0}{\|x_0\|}\\
  \addlinespace
  0
  \end{bmatrix}.
  \end{align*}
\end{proof}

\section{Projections onto hyperbola $\Ct_{\gamma}$ and bilinear constraint set $C_{\gamma}$ when $\gamma<0$}
\label{sec:negative:gam}
Armed with the results in Sections~\ref{section:hyperbola} and ~\ref{sec:positive:gam},
we can study $P_{\Ct_{\gamma}}$ and
$P_{C_{\gamma}}$ when $\gamma<0$.
Define
$$T_{1}: X\times X\rightarrow X\times X: (x,y)\mapsto (y,x),\text{ and }$$
$$T_{2}:X\times X\rightarrow X\times X: (x,y)\mapsto (x,-y).$$
\begin{theorem}
\label{t:20114711}
Let $\gamma<0, \tu_0,\tv_0 \in \Xt$, and $\Ct_{\gamma} =
      \menge{(u,v)\in \Xt \times \Xt}{\norm{u}^2 -\norm{v}^2 = 2\gamma}$.
Then the following hold:
\begin{enumerate}
\item \label{t:20114711.a} When $\tu_0 \neq 0, \tv_0 \neq 0$, we have
$$\proj{\Ct_{\gamma}}(\tu_0,\tv_0)
  =
  \left\{\begin{bmatrix}
   \frac{\tu_0}{1-\lambda}\\
   \addlinespace
   \frac{\tv_0}{1+\lambda}
  \end{bmatrix}
  \right\}
  $$
 in which $\lambda$ is the unique root of $H(\lambda)=0$ in $\left]-1,1\right[$,
  where \begin{align*}
     H(\lambda) =
      \frac{(\lambda^2 + 1 ) p - 2\lambda q}{2(1-\lambda^2)^2} +\gamma,~
      p =\norm{\tv_0}^2 - \norm{\tu_0}^2,~\text{ and } q =\norm{\tu_0}^2 + \norm{\tv_0}^2.
      \end{align*}
      \item \label{t:20114711.c} When $\tv_0 = 0$, we have
    \begin{align}
    \proj{\Ct_{\gamma}}(\tu_0,0) =
    \left\{
    \begin{bmatrix}
    \frac{\tu_0}{2}\\
    \tv
    \end{bmatrix}
    \bigg|\ \|\tv\|^2 =  \frac{\|\tu_0\|^2}{4}-2\gamma\right\}.
    \end{align}
      \item \label{t:20114711.b}
    a) When $\tu_0 = 0$ and $ \|\tv_0\|\geq 2\sqrt{2(-\gamma)}$, we have
    \begin{align}
    \proj{\Ct_{\gamma}}(0,\tv_0) = \left\{
    \begin{bmatrix}
    \tu\\
    \frac{\tv_0}{2}
    \end{bmatrix}
   \bigg|\
    \|\tu\|^2 = \frac{\|\tv_0\|^2}{4}+2\gamma\right\}.
    \end{align}
    b) When $\tu_0 = 0$ and $0<\|\tv_0\|< 2\sqrt{2(-\gamma)}$, we have
    \begin{align}
    \proj{\Ct_{\gamma}}(0,\tv_0) = \left\{\begin{bmatrix}
    0\\
    \sqrt{2(-\gamma)}\frac{\tv_0}{\|\tv_0\|}\end{bmatrix}
    \right\}.
    \end{align}
\end{enumerate}
\end{theorem}
\begin{proof}
Since
$$\minimize ~ \|\tu-\tu_{0}\|^2+\|\tv-\tv_{0}\|^2 \quad \text{subject to}\quad \norm{\tu}^2 -\norm{\tv}^2= 2\gamma$$
is equivalent to
$$\minimize ~ \|\tv-\tv_{0}\|^2+
\|\tu-\tu_{0}\|^2 \quad \text{subject to}\quad \norm{\tv}^2 -\norm{\tu}^2= 2(-\gamma),$$
we have
$P_{\Ct_{\gamma}}=T_{1}P_{\Ct_{-\gamma}}T_{1}.$ It suffices to apply
\cref{t:201147}.
\end{proof}

\begin{theorem}
\label{t:negativegamma}
Let $\gamma<0,~x_0,y_0 \in X$, and $C_{\gamma} =
      \menge{(x,y)\in X \times X}{\scal{x}{y} = \gamma}$.  
Then the following hold:
\begin{enumerate}
    \item
    When $x_0 \neq  \pm y_0$, the projection is a singleton:
    \begin{align*}
\proj{C_{\gamma}}(x_0,y_0)
          =
          \left\{\begin{bmatrix}
    \frac{x_0+\lam y_0}{1-\lam^2} \\
    \addlinespace
    \frac{y_0+\lam x_0}{1-\lam^2}
    \end{bmatrix}
    \right\}
     \end{align*}
    in which  $\lambda$ is the unique solution of $H(\lambda)=0$ in
    $\left]-1,1\right[$, where
    \begin{align*}
      H(\lambda) =
      \frac{(\lambda^2 + 1 ) p - 2\lambda q}{2(1-\lambda^2)^2} +\gamma,~
      p =-2\scal{x_0}{y_0},~\text{ and } q =\norm{x_0}^2 + \norm{y_0}^2.
  \end{align*}
   \item
    When $x_0 =y_0$,  the projection is a set:
    $$\proj{C_{\gamma}}(x_0,x_0) =\left\{
    \begin{bmatrix}
    \frac{x_{0}}{2}+\frac{\tu}{\sqrt{2}}\\
    \addlinespace
    \frac{x_{0}}{2}-\frac{\tu}{\sqrt{2}}
    \end{bmatrix}
    \bigg|\ \|\tu\|^2=-2\gamma+ \frac{\|x_0\|^2}{2},~~ \tu \in X\right\}.$$

    \item
    a) When $x_0 =-y_0$ and $\|x_0\|\geq 2\sqrt{-\gamma}$,
    the projection is a set:
    $$\proj{C_{\gamma}}(x_0,-x_0) =  \left\{
    \begin{bmatrix}
    \frac{x_{0}}{2}-\frac{\tv}{\sqrt{2}}\\
    \addlinespace
    \frac{-x_{0}}{2}-\frac{\tv}{\sqrt{2}}
    \end{bmatrix}
    \bigg|\ \|\tv\|^2=\frac{\|x_0\|^2}{2}+2\gamma,
    ~~ \tv \in X\right\}.$$
  b) When $x_0 =-y_0$ and $0<\|x_0\|< 2\sqrt{-\gamma}$, the projection
  is a singleton:
    $$\proj{C_1}(x_0,-x_0) =  \left\{
    \sqrt{-\gamma}
    \begin{bmatrix}
    \frac{x_{0}}{\|x_{0}\|}\\
    \addlinespace
    \frac{-x_{0}}{\|x_{0}\|}
    \end{bmatrix}
    \right\}.$$
   \end{enumerate}
\end{theorem}
\begin{proof} Since
$$\minimize ~ \|x-x_{0}\|^2+\|y-y_{0}\|^2 \quad \text{subject to}\quad \scal{x}{y}= \gamma$$
is equivalent to
$$\minimize ~ \|x-x_{0}\|^2+\|z-(-y_{0})\|^2 \quad \text{subject to}\quad \scal{x}{z}= -\gamma,$$
we have
$P_{C_{\gamma}}=T_{2}P_{C_{-\gamma}}T_{2}$. It remains to apply Theorem~\ref{t:withgamma}.
\end{proof}

\section*{Acknowledgments}
HHB and XW were supported by NSERC Discovery grants. MKL was supported by SERB-UBC fellowship and HHB and XW's NSERC Discovery grants. This work was presented at the SIAM Conference on Optimization in July 20--23, 2021 by MKL.

\end{document}